\documentclass[12pt,a4paper,oneside,onecolumn]{article}

\usepackage[english]{babel}
\usepackage{hyperref}

\usepackage{todonotes}
\usepackage{amsthm}
\usepackage{amsmath}
\usepackage{amsfonts}
\usepackage{natbib}[2009/07/16 8.31 (PWD, AO)]
\usepackage{enumitem}
\usepackage{graphicx}
\usepackage{appendix}
\usepackage{xcolor}
\usepackage[utf8]{inputenc}
\usepackage{authblk}
\title{Exterior Shifting of Low Genus Surfaces}
\author[1]{Aaron Keehn\thanks{\href{mailto:aaron.keehn@mail.huji.ac.il}{aaron.keehn@mail.huji.ac.il}. }}
\author[1]{Eran Nevo\thanks{\href{mailto:nevo@math.huji.ac.il}{nevo@math.huji.ac.il}. \\
Both authors were partially supported by the Israel Science Foundation grant ISF-2480/20.}}
\affil[1]{Einstein Institute of Mathematics,
 Hebrew University, Jerusalem~91904, Israel}
\newtheorem{theorem}{Theorem}[section]
\newtheorem{lemma}[theorem]{Lemma}
\newtheorem{proposition}[theorem]{Proposition}
\newtheorem{corollary}[theorem]{Corollary}
\newtheorem{conjecture}[theorem]{Conjecture}
\theoremstyle{definition}
\newtheorem{definition}[theorem]{Definition}
\theoremstyle{definition}
\newtheorem{remark}[theorem]{Remarks}
\newcommand\overmat[2]{%
  \makebox[0pt][l]{$\smash{\color{white}\overbrace{\phantom{%
    \begin{matrix}#2\end{matrix}}}^{\text{\color{black}#1}}}$}#2}

\newcommand\partialphantom{\vphantom{\frac{\partial e_{P,M}}{\partial w_{1,1}}}}

\begin{document}
\maketitle

\begin{abstract}
    We characterize the possible exterior shiftings of \(K\), where \(K\) runs over all triangulation of the torus, or the projective plane, or the Klein bottle. Further, we give a deterministic polynomial-time algorithm for computing the exterior shifting of a given triangulation  \(K\) as above.
\end{abstract}
\section{Introduction}

Kalai \cite{KalaiPart1, KalaiPart2} introduced the (exterior) algebraic shifting operator on simplicial complexes, which is a canonical way, after fixing a field, of associating a shifted complex to a given simplicial complex, based on the exterior face ring. This operator has found many applications in Combinatorics, especially in $f$-vector theory, and is interesting on its own, see e.g. Kalai's survey
\cite{KalaiShifting}.  Computing the algebraic shifting 
\(\Delta (K)\) of an input simplicial complex \(K\) is an instance of the general problem of symbolic determinant identity testing. 
Gauss elimination provides an efficient randomized algorithm, but the known deterministic algorithms run in super-exponential time in the input size. 

In this paper we always shift over a field of characteristic zero; for all such fields the shifting operators coincide, as is well known and follows easily from the definition. 
Note that \(\Delta (K)\) is constant for all $n$-vertex triangulations \(K\) of the $2$-sphere; see the discussion in \cite[Sec.2.4]{KalaiShifting}. Indeed,  \(\Delta (K)\) contains exactly the triangles $\{1,3,n\}$ and $\{2,3,4\}$ and all the faces they imply by \(\Delta (K)\) being a shifted simplicial complex.
This raises a question: what can be said for algebraic shifting of triangulations of other surfaces?
    
Recently Bulavka, Nevo and Peled \cite{bulavka2023volume} showed that  \(\left\{1,3,n\right\}\in \Delta(K)\) for every $n$-vertex triangulation of the torus, the projective plane and the Klein bottle. They conjectured that the same is true for any triangulation of a compact surface without boundary.  

First, we generalize their result by providing a polytime algorithm to compute the exterior shifting for any triangulation of each of the above surfaces.

\begin{theorem}\label{alg exists}
Let \(K\) be a triangulation of either the torus, or 
the projective plane \(\mathbb{RP}^{2}\), or the Klein bottle, with \(n\) vertices. Then there exists a deterministic algorithm for calculating the exterior shifting \(\Delta(K)\), that runs in \(O(n^{p})\) time, where \(p \in \mathbb{N}\) depends only on the topology of the surface \(\left|K\right|\).
\end{theorem}

For the proof, we study the effect of edge contractions and vertex splits on the exterior shifting of a triangulation. The key ingredient is defining what we call \emph{critical regions}, which are subcomplexes of a triangulation in which contracting an edge (or splitting a vertex) affects the shifting in a predictable way.
    This allows us to carefully choose edges to contract, and thus the problem reduces to calculating the shifting of triangulations with few vertices.
    
    To do this, we made use of the lists of all irreducible triangulations for the relevant surfaces (\cite{LavrenchenkoTorus, BarnetteRP2, LawrencenkoNegamiKlein, SulankeNoteKlein}) to calculate the exterior shifting for a finite number of cases, generated from the irreducible ones by a bounded number of vertex splits. 
    
Second, the same ideas are used to obtain a characterization of the possible shiftings for these surfaces: 
let \(\leq _{p}\) denote the partial ordering on equal-sized subsets of \(\mathbb{N}\), where \(\left\{ a_{1}<\cdots<a_{m}\right\} \leq_{p}\left\{ b_{1}<\cdots<b_{m}\right\} \) if and only if \(a_{i}\leq b_{i}\)  for all \(1\leq i\leq m\).

    \begin{theorem}
    \label{possible shiftings}
        (i) If \(K\) is a triangulation of the torus with \(n\) vertices, then the possible sets of maximal faces of \(K\) (w.r.t. \(\leq_{p}\)) are exactly the following:
        \[\mathcal{T}_{1} := \left\{ \left(2,3,4\right),\left(1,3,n\right),\left(1,4,7\right),\left(1,5,6\right),\left(3,n\right),\left(6,7\right)\right\} ,\]
        \[\mathcal{T}_{2} := \left\{ \left(2,3,4\right),\left(1,3,n\right),\left(1,4,8\right),\left(3,n\right),\left(4,8\right),\left(5,7\right)\right\} ,\]
        \[\mathcal{T}_{3} := \left\{ \left(2,3,4\right),\left(1,3,n\right),\left(1,4,8\right),\left(3,n\right),\left(4,9\right),\left(5,6\right)\right\} ,\]
        \[\mathcal{T}_{4} := \left\{ \left(2,3,4\right),\left(1,3,n\right),\left(1,4,8\right),\left(3,n\right),\left(4,10\right)\right\} .\]
        (ii) If \(K\) is a triangulation of \(\mathbb{RP}^{2}\) with \(n\) vertices,  then the possible sets of maximal faces of \(K\) (w.r.t. \(\leq_{p}\)) are exactly the following:
        \[\mathcal{P}_{1} := \left\{ \left(1,3,n\right),\left(1,5,6\right),\left(3,n\right),\left(5,6\right)\right\} ,\]
        \[\mathcal{P}_{2} := \left\{ \left(1,3,n\right),\left(1,4,7\right),\left(3,n\right),\left(4,7\right)\right\} .\]
        (iii) If \(K\) is a triangulation of the Klein bottle with \(n\) vertices,  then the possible sets of maximal faces of \(K\) (w.r.t. \(\leq_{p}\)) are exactly the following:
        \[\mathcal{KB}_{1} := \left\{ \left(1,3,n\right),\left(1,4,8\right),\left(1,5,6\right),\left(3,n\right),\left(4,8\right),\left(5,7\right)\right\} ,\]
        \[\mathcal{KB}_{2} := \left\{ \left(1,3,n\right),\left(1,4,9\right),\left(3,n\right),\left(4,9\right),\left(5,6\right)\right\} ,\]
        \[\mathcal{KB}_{3} := \left\{ \left(1,3,n\right),\left(1,4,9\right),\left(3,n\right),\left(4,10\right)\right\} .\]
    \end{theorem}

    These characterizations leads us to conjecture the following for a triangulation of every compact surface without boundary ($<_{\mathrm{lex}}$ stands for the lexicographic order, which is a linear extension of $\leq_{p}$):
    \begin{conjecture}
    \label{conjecture on edges in shifting}
    If \(K\) is a triangulation of a compact surface without boundary, and \(\sigma,\sigma'\in\Delta\left(K\right)_{1}\) are edges with \(\sigma<_{\mathrm{lex}}\sigma'\), then:
    \[\left\{ 1\right\} \cup\sigma\notin\Delta\left(K\right)\Rightarrow\left\{ 1\right\} \cup\sigma'\notin\Delta\left(K\right).\]
    \end{conjecture}
    Note that the condition \(\left\{ 1\right\} \cup\sigma\notin\Delta\left(K\right)\) means that \(\sigma\) contributes to the 1st homology of \(\Delta\left(K\right)\), see Theorem~\ref{Shifting preserves Betti numbers}.
    Of course, replacing $<_{\mathrm{lex}}$ by $\leq_{p}$ everywhere in the conjecture makes the assertion trivial, as \(\Delta\left(K\right)\) is shifted.

    \textbf{Outline.} In Section \ref{Preliminaries section} we establish some terminology and give the necessary background on exterior shifting and combinatorial topology. In Section \ref{Commutativity of vertex splits section} we discuss the combinatorics of vertex splitting and how vertex splits commute with each other, and in Section \ref{How do vertex splits affect shifting} we consider how vertex splits affect the shifting of a simplicial complex. In Section \ref{critical regions section} we define critical regions and prove their crucial properties. In Section \ref{algorithm section} we prove Theorems~\ref{alg exists} and~\ref{possible shiftings} for the torus. For the sake of readability, we leave the discussion of the projective plane and the Klein bottle to the Appendix.
\section{Preliminaries}
\label{Preliminaries section}
\subsection{Combinatorial Topology}
We will denote the realization as a topological space of a simplicial complex \(K\)  by \(\left|K\right|\).

\begin{definition}
    
    A pure 2-dimensional simplicial complex \(S\) will be called \emph{surface-embeddable} if there is some triangulated surface \(\widetilde{S}\) such that \(S\) is a subcomplex of \(\widetilde{S}\). We denote by \(\partial S\) the subgraph of \(S\) where \(\left|\partial S\right|\subseteq\left|S\right|\) is the boundary of \(\left|S\right|\).
 
    A vertex of \(S\) is a \emph{boundary vertex} if it lies in \(\partial S\), otherwise it is an \emph{internal vertex}.

    An edge of \(S\) is a \emph{boundary edge} if it lies in \(\partial S\), and it is an \emph{internal edge} if it is incident to an internal vertex. All edges which are neither internal nor boundary are \emph{diagonals} (in other words, all edges not in the boundary that intersect the boundary at both vertices).
\end{definition}
\begin{definition}    
    Let \(K\) be a simplicial complex, and \(a,b,c\) vertices in \(K\). If \(K\) contains the edges \(ab,ac,bc\) but not the 2-simplex \(abc\), we say that \(abc\) is a \emph{missing triangle} in \(K\).
\end{definition}

\begin{definition}    
    For a simplicial complex \(K\) and an edge \(vv'\) in \(K\), we have an operation of \emph{edge contraction} where we identify the two vertices \(v\) and \(v'\), and remove any duplicate faces, to obtain a new simplicial complex \(K'\). The inverse operation that starts with \(K'\) and gives \(K\) is called \emph{vertex splitting}.

    Since we will be dealing with triangulations of surfaces, we will only be interested in contracting edges which preserve the topology of the complex, which happens precisely when the contracted edge does not belong to any missing triangle\footnote{The only exception to this is the boundary of a tetrahedron, where there are no missing triangles, and at the same time no contractible edges.}. We shall call such an edge \emph{contractible}. 
    
    A triangulation of a surface in which every edge is part of a missing triangle is called \emph{irreducible}.

    If \(K\) is a triangulation of a surface, and \(v_{1}v_{2}\) a contractible edge in \(K\), let \(K'\) be the triangulation obtained by contracting \(v_{1}v_{2}\) to a vertex \(v\). Then there are two 2-faces \(uv_{1}v_{2},wv_{1}v_{2}\) in \(K\) containing the edge \(v_{1}v_{2}\), and in this case we say that the vertex \(v\) in \(K'\) is \emph{split between} \(u\) and \(w\) to obtain \(K\).

    The \emph{link} of a vertex $v\in K$ is denoted by $\mathrm{lk}_{v}(K)$.
\end{definition}

\subsection{Algebraic Shifting}

\begin{definition}
    We have a partial ordering \(\leq _{p}\) on subsets of \(\mathbb{N}\), where \[\left\{ a_{1}<\cdots<a_{m}\right\} \leq_{p}\left\{ b_{1}<\cdots<b_{m}\right\} \] if and only if: \[\forall1\leq i\leq m\ \ \ a_{i}\leq b_{i}.\]

    A simplicial complex that is closed downward w.r.t. \(\leq_{p}\) is called a \emph{shifted} complex.
\end{definition}
We recall Kalai's definition of exterior algebraic shifting:
Let \(\left(e_{i}\right)_{i\in\left[n\right]}\) be the standard basis for \(\mathbb{R}^{n}\), and let \(X=\left(x_{ij}\right)_{1\leq i,j\leq n}\) be a generic matrix over \(\mathbb{R}\), i.e. all the entries of \(X\) are algebraically independent over \(\mathbb{Q}\). We define a generic basis \(f_{i}=\sum_{j=1}^{n}x_{ij}e_{j}\) of \(\mathbb{R}^{n}\), where $1\leq i\leq n$.

Given a simplicial complex \(K\) with \(n\) vertices, we have the \emph{exterior face ring}: \[ \bigwedge K=\bigwedge\mathbb{R}^{n}/\left(e_{\sigma}|\ \sigma\notin K\right),\] where \(e_{\left\{ i_{1}<\cdots<i_{k}\right\} }:=e_{i_{1}}\wedge\cdots\wedge e_{i_{k}}\). Denote by \(q:\bigwedge\mathbb{R}^{n}\rightarrow\bigwedge K\) the quotient map. Given a partial order \(<\) on the power set of \(\left[n\right]\), define:
\[\Delta^{<}\left(K\right)=\left\{ \sigma\subseteq\left[n\right]|\ q\left(f_{\sigma}\right)\notin\text{span}_{\mathbb{R}}\left\{ q\left(f_{\tau}\right)|\ \tau<\sigma,\ \left|\tau\right|=\left|\sigma\right|\right\} \right\} .\]
In \cite{BjornerKalai}, Bj\"{o}rner and Kalai showed that \(\Delta^{<}\left(K\right)\) is a shifted simplicial complex with the same f-vector as \(K\), for every total order \(<\) that refines \(\leq_{p}\). Moreover, \(\Delta^{<}\left(K\right)\) does not depend on the labeling of the vertices.

We shall denote by \(\Delta\) the operation \(\Delta^{\mathrm{lex}}\), where \(\mathrm{lex}\) is the lexicographic order (which indeed refines \(\leq_{p}\)).

For \(n\geq k \geq d \geq 1\), define a linear transformation: 
\[\psi_{n,k}^{d}:\bigoplus_{i=d}^{k}\bigwedge^{1}\mathbb{R}^{n}\rightarrow\bigwedge^{d+1}\mathbb{R}^{n}\]
\[\psi_{n,k}^{d}\left(m_{d},\ldots,m_{k}\right)=\sum_{i=d}^{k}f_{\left[d-1\right]\cup\left\{ i\right\} }\wedge m_{i},\]
and for a simplicial complex \(K\) with \(n\) vertices, define a composition of maps:
\[\psi_{K,k}^{d}:\bigoplus_{i=d}^{k}\bigwedge^{1}\mathbb{R}^{n}\stackrel{\psi_{n,k}^{d}}{\rightarrow}\bigwedge^{d+1}\mathbb{R}^{n}\rightarrow\bigwedge^{d+1} K.\]

\begin{proposition} (\cite[Lemma 2.4]{bulavka2023volume})
\label{image of psi}
    \[\mathrm{Im} \psi_{n,k}^{d} = \mathrm{span} \left\{ f_{\sigma}|\ \sigma\leq_{p}\left\{ 1,2,\ldots,d-1,k,n\right\} \right\} .\]
\end{proposition}

\begin{corollary}
    We have:
    \[\mathrm{corank} \psi_{K,k}^{1} = \left|\mathrm{Tail}_{\mathrm{lex}}\left(\Delta (K),\left\{ k+1,k+2\right\} \right)\right|,\]
    and:
    \[\mathrm{corank} \psi_{K,k}^{2} = \left|\mathrm{Tail}_{\mathrm{lex}}\left(\Delta (K),\left\{ 1,k+1,k+2\right\} \right)\right|,\]
    where \(\mathrm{Tail}_{\mathrm{lex}}\left(S,\sigma\right):=\left\{ \delta\in S|\ \left|\delta\right|=\left|\sigma\right|,\ \delta\geq_{\mathrm{lex}}\sigma\right\}\) for a (shifted) simplicial complex \(S\).
\end{corollary}
\begin{proof}
    By the definitions of \(\Delta(K)\) and \(\psi_{K,k}^{d}\) 
    , the corank of \(\psi_{K,k}^{d}\) is the number of \(d\)-faces \(\sigma\) in \(\Delta(K)\) such that \(f_{\sigma}\notin \mathrm{Im} \psi_{n,k}^{d}\). Denote 
    \[\mu_{d,k,n}:=\left\{ 1,2,\ldots,d-1,k,n\right\}. \]
    Then by \ref{image of psi}, for \(\sigma\in{\left[n\right] \choose d+1}\cap \Delta(K)\) we have:
    \[f_{\sigma}\notin \mathrm{Im} \psi_{n,k}^{d}\iff\sigma\not\leq_{p}\mu_{d,k,n}.\]
    Moreover, for \(d=1\) we have:
    \[\sigma\not\leq_{p}\mu_{1,k,n}=\left\{ k,n\right\} \iff\sigma\geq_{\mathrm{lex}}\left\{ k+1,k+2\right\}, \]
    and similarly for \(d=2\):
    \[\sigma\not\leq_{p}\mu_{2,k,n}=\left\{ 1,k,n\right\} \iff\sigma\geq_{\mathrm{lex}}\left\{ 1,k+1,k+2\right\}. \]
\end{proof}
We can represent \(\psi_{K,k}^{d}\) by a \(f_{d}(K)\times (k-d+1)n\) matrix, where $f_i(K)$ denotes the number of $i$-dimensional faces in $K$. First, denote by \(e_{i,v}\) the element of \(\bigoplus_{i=d}^{k}\bigwedge^{1}\mathbb{R}^{n}\) with \(e_{v}\) in the \(i\)-th coordinate, and \(0\) in all the other coordinates. By defining an inner product on \(\bigwedge^{d+1}\mathbb{R}^{n}\) with orthonormal basis \(\left\{ e_{\sigma}|\ \sigma\in{\left[n\right] \choose d+1}\right\} \), we get \(\left\langle e_{\sigma},\psi_{n,k}^{d}\left(e_{i,v}\right)\right\rangle =0\) if \(v\notin \sigma\). 

Otherwise, write \(\sigma=\left\{ v_{0}<v_{1}<\ldots<v_{d}\right\}  \), with \(v=v_{j}\) for some \(0\leq j\leq d\). Then \(\left\langle e_{\sigma},\psi_{n,k}^{d}\left(e_{i,v}\right)\right\rangle = \left(-1\right)^{j+1} \det (M)\), where \(M\) is the submatrix of \(X\) whose rows are given by \(\sigma\backslash\left\{ v\right\} \), and whose columns are given by \(\left[d-1\right]\cup\left\{ i\right\}\).

The following known results will be of use to us later.

    \begin{theorem}[Shifting preserves Betti numbers; Bj\"{o}rner-Kalai~\cite{BjornerKalai, KalaiShifting}]
    \label{Shifting preserves Betti numbers}
        For every simplicial complex \(K\), and every \(d\geq 0\):
        \[b_{d}\left(K\right)=b_{d}\left(\Delta\left(K\right)\right)=\left|\left\{ \sigma\in\Delta\left(K\right)_{d}|\ 1\notin\sigma\ and\ \left\{ 1\right\} \cup\sigma\notin\Delta\left(K\right)_{d}\right\} \right|,\]
        where \(b_{d}\) is the \(d\)-th Betti number.
    \end{theorem}

    \begin{theorem}[Shifting a union over a simplex; Nevo~\cite{nevo2006algebraic}]
    \label{Shifting a union over a simplex}
        Let \(K\) and \(L\) be simplicial complexes such that \(K\cap L=\left\langle \sigma\right\rangle \) is a simplex (and all its subsets). 
        Let \(\left(K\cup L\right)_{0}=\left[n\right]\), and let \(\left[n\right]\supseteq T=\left\{ t_{1}<\cdots<t_{j}<t_{j+1}\right\} \). Then:
        \[T\in\Delta\left(K\cup L\right)\iff t_{j+1}-t_{j}\leq D_{K}\left(T\right)+D_{L}\left(T\right)-D_{\left\langle \sigma\right\rangle }\left(T\right),\]
        where for a simplicial complex \(S\): 
        \[D_{S}\left(T\right):=\left|\left\{ s\in\left[n\right]|\ s>t_{j}\ and\ \left\{ s\right\} \cup\left(T\backslash\left\{ t_{j+1}\right\} \right)\in\Delta\left(S\right)\right\} \right|.\]
    \end{theorem}

    Another important and recent result\footnote{This is not how the result is formulated in \cite{bulavka2023volume}, but it is equivalent, since for any \(\sigma\in{\left[n\right] \choose 3}\), \(\sigma\leq_{p}\left\{ 1,3,n\right\}\) if and only if  
    \(\sigma\leq_{\mathrm{lex}}\left\{ 1,3,n\right\}\).} that we will need is:
    \begin{theorem}[Bulavka-Nevo-Peled~\cite{bulavka2023volume}]
    \label{volume rigidity of torus kb and rp2}
    For every triangulation \(K\) of the torus, the Klein bottle or the projective plane, with \(\left|K_{0}\right|=n\), we have \(\left\{ 1,3,n\right\} \in\Delta\left(K\right)\).
    \end{theorem}
\section{Vertex splitting and edge contraction}
\subsection{Commutativity of vertex splits}
\label{Commutativity of vertex splits section}

\begin{definition}    
    A triangulation \(K\) of a surface that contains no missing triangles which enclose a disk shall be called \emph{prime}.

\end{definition}

\begin{definition}
    Let \(G\) be a subgraph of a surface-embeddable \(K\), and let \(K\leftarrow K'\) be an edge contraction such that an edge \(v_{1}v_{2}\) in \(K'\) contracts to \(v\), with \(v \in G\). We say that \(G\) \emph{survives} the vertex split \(K\rightarrow K'\) if for some \(v' \in \left\{ v_{1},v_{2} \right\}\) the edge \(xv'\) is in \(K'\) for all neighbors \(x\) of \(v\) in \(G\). Otherwise, the split \emph{destroys} \(G\).
    
    For a sequence of splits from \(K_{0}=K\) to \(K_{n}=K'\), \(G\) survives the sequence if it survives each individual split; if \(v'\) is the vertex that witnesses the survival of \(G\) after a single split at a vertex \(v \in G\) (as in the above definition), we rename \(v'\) to be \(v\) and proceed inductively. Otherwise we say that the sequence destroys \(G\).
\end{definition}

We now give a generalized version of~\cite[Prop. 4.1]{nevo2022vertex}.
\begin{lemma} [Commutativity]
\label{commutativity}
    Let \(K=K(0)\) be a triangulation of a surface, and assume \((u,v,w)\) is a simple path in \(K\) such that \(uvw\) is not a simplex in \(K\). Assume we have a sequence of vertex splits:
    \[K=K({0})\rightarrow K({1})\rightarrow\cdots\rightarrow K({n-1})\rightarrow K({n})=K'\]
    such that \((u,v,w)\) does not survive. Then there exists a sequence of vertex splits from \(K\) to \(K'\) where \((u,v,w)\) is destroyed on the first split.
\end{lemma}
\begin{proof}
 Assume the path is destroyed on the split \(K({i-1})\rightarrow K({i})\) for some \(i \geq 2\). The only way for the split to destroy \((u,v,w)\) is to split \(v\) between two vertices \(a\) and \(b\) such that \(a\) and \(b\) separate \(u\) and \(w\) in the cycle \(\mathrm{lk}_{v}\left(K({i-1})\right)\). In particular, \(a\) and \(b\) are not adjacent in \(\mathrm{lk}_{v}\left(K({i-1})\right)\). Denote by \(v_{1},v_{2}\) the vertices in \(K({i})\) that were split from \(v\), i.e. \(v_{1}v_{2}\) is the edge in \(K({i})\) that contracts to give \(K({i-1})\).
    
    Now, let \(\epsilon\) be the edge that is contracted in \(K({i-1})\) to get \(K({i-2})\). Since a vertex split can create a new missing triangle only when splitting a vertex \(x\) between two vertices \(y,z\) that are adjacent in the link of \(x\), this means that \(\epsilon\) is not part of a missing triangle in \(K({i})\) and thus can be contracted to obtain another triangulation \(K({i-1})'\). (Here we abused notation and denoted a preimage of \(\epsilon\) in $K(i)$, under the map induced by contracting $v_1v_2$, also by \(\epsilon\).)  We will show that \(v_{1}v_{2}\) can be contracted in \(K({i-1})'\), and thus will be done by induction.
Indeed, by contradiction assume not. Then the edges \(\epsilon\) and $v_1v_2$ must belong to an induced $4$-cycle \(c_{1}c_{2}v_{1}v_{2}\) in \(K({i})\). But then, contracting $v_1v_2$ would make \(\epsilon\) part of the missing triangle $c_{1}c_{2}v$ in $K(i-1)$, a contradiction.
\end{proof}

\begin{corollary}
\label{prime triangulation survival}
    Let \(K\) be a prime triangulation of a surface, and \(K'\) be a prime triangulation of the same surface obtained from \(K\) via a sequence of vertex splits. Then there exists a sequence of vertex splits from \(K\) to \(K'\) such that all intermediate triangulations are also prime.
\end{corollary}
\begin{proof}
Assume we have a sequence of vertex splits:

    \[K=K({0})\rightarrow K({1})\rightarrow\cdots\rightarrow K({n-1})\rightarrow K({n})=K'\]
such that some intermediate triangulation is not prime. By induction it is enough to assume that \(K({1})\) is not prime, and thus contains a missing triangle \(ABC\) enclosing a disk. Since \(K\) is prime, this happens precisely when \(ABC\) is a 2-face in \(K\), and the vertex \(A\) is split between \(B\) and \(C\) to obtain \(K({1})\). (Again we abuse notation, which we will repeat, and denote by $A$ both the split vertex and one chosen copy of it after the split.) Since \(K'\) is prime, the $3$-cycle \(ABC\) does not survive all the splits, and we can therefore assume by Lemma \ref{commutativity} that \(ABC\) was destroyed on the second split. In this case, the split \(K({0})\rightarrow K({1})\) adds a new vertex \(O\) to the interior of \(ABC\), and \(K({1})\rightarrow K({2})\) splits the vertex \(A\) (without loss of generality) between \(O\) and a vertex \(V\) lying outside of the disc bounded by \(ABC\) and containing $O$. By induction, it is enough to show that we can replace these two vertex splits with \(K({0})\rightarrow K({1})'\rightarrow K({2})\) where \(K'({1})\) is prime.

    \begin{figure}
        \centering
        \includegraphics[scale=0.7]{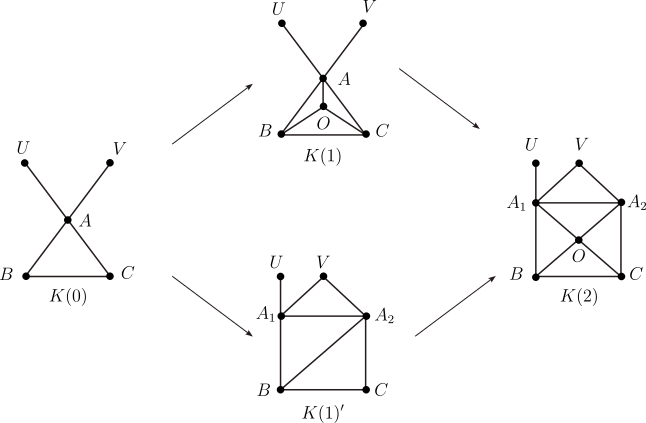}
        \caption{\(K({2})\) can be reached from \(K({0})\) through either \(K({1})\) or \(K({1})'\)}
        \label{fig: commutativity}
    \end{figure}
    
    Indeed, since \(K\) is prime, the degree of \(A\) in \(K\) must be at least \(4\), meaning there exists another neighbor \(U\) of \(A\), other than \(V,B,C\). Let us denote by \(A_{1},A_{2}\) the vertices in \(K({2})\) split from \(A\), so that \(A_{1}B,A_{2}C,A_{1}U\in K({2})\). 
    
    Let \(K({1})'\) be obtained by contracting the vertex \(O\) to \(B\). Obviously, contracting \(A_{1}A_{2}\) in \(K({1})'\) gives \(K\), equivalently,  \(K({1})'\) is obtained from $K$ by vertex split at $A$ between $B$ and $V$. It is enough to see that no vertex in \(K({1})'\) has degree 3 for it to be prime, since \(K\) is already prime. 
    Now, clearly the degrees of all vertices but $A_1$ and $A_2$ can only increase under the split at $A$, so are at least 4.
    Further, 
    \(A_{1}\) neighbors \(A_{2},B,V,U\) while \(A_{2}\) neighbors \(A_{1},B,C,V\) in \(K({1})'\), thus indeed \(K({1})'\) is prime. (See Figure \ref{fig: commutativity}.)
\end{proof}

\subsection{How do vertex splits affect shifting?}
\label{How do vertex splits affect shifting}
It turns out that splitting a vertex will not cause the tail of the shifting to grow, so to speak:
\begin{proposition}
\label{vert split preserves tails}
    Let \(K\) be a triangulated surface, and \(K'\) obtauned from \(K\) by a vertex split at $v$. Then for all \(m \geq 3\):
    \[\left|\mathrm{Tail}_{\mathrm{lex}}\left(\Delta\left(K\right),\left( m+1,m+2\right) \right)\right|\geq\left|\mathrm{Tail}_{\mathrm{lex}}\left(\Delta\left(K'\right),\left( m+1,m+2\right) \right)\right|,\]
    and:
    \[\left|\mathrm{Tail}_{\mathrm{lex}}\left(\Delta\left(K\right),\left( 1,m+1,m+2\right) \right)\right|\geq\left|\mathrm{Tail}_{\mathrm{lex}}\left(\Delta\left(K'\right),\left( 1,m+1,m+2\right) \right)\right|.\]
\end{proposition}

\begin{proof}
    We must show that for \(i=1,2\) we have:
    \[\mathrm{corank}\psi_{K',m}^{i}\leq\mathrm{corank}\psi_{K,m}^{i}.\]
    
    For \(i=1\), note that in the matrix of \(\psi_{K,m}^{1}\), all the rows corresponding to edges containing \(v\) 
    are linearly independent. Therefore, we can extend to a basis for the row space of the matrix to get a 
    subgraph
    \(E\subseteq K\)  such that \(E\) contains all the edges containing \(v\), and \(\psi_{E,m}^{1}\) is of full rank, and its rank is equal to that of \(\psi_{K,m}^{1}\). 
    The vertex split at $v$ in $K$ restricted to $E$ yields a subgraph $E'$ of $K'$ with 3 more edges than in $E$; and recall that $3\leq m$. Thus, by~\cite[Lemma 3.3.1]{nevo2007algebraic}, we get that \(\psi_{E',m}^{1}\) is of full rank, so its rank equals 3 plus the rank of \(\psi_{K,m}^{1}\). As $K'$ has 3 more edges than $K$, we conclude that $\mathrm{corank}\psi_{K',m}^{1}\leq\mathrm{corank}\psi_{K,m}^{1}$, as desired.

    For \(i=2\), since \(\psi_{K',m}^{2}\) has precisely two more rows than \(\psi_{K,m}^{2}\), we want to show that the rank of \(\psi_{K',m}^{2}\) is at least 2 more than the rank of \(\psi_{K,m}^{2}\). The proof in this case is the same as the beginning of the proof of \cite[Lemma 3.1]{bulavka2023volume}:
    
    If the edge we are contracting is \(\left\{u<w\right\}\), we assume without loss of generality that $u$ and $w$ are the first of the \(n\) vertices of \(K'\). Then we substitute all \(x_{wi}\) with \(x_{ui}\) in the matrix of \(\psi_{K',m}^{2}\) to get a matrix \(A'\), and then add all columns of \(A'\) corresponding to \(u\) to the  corresponding columns of \(w\) to get a matrix \(A\). Clearly the rank of \(\psi_{K',m}^{2}\) is greater or equal to that of \(A\). Moreover, the two rows of \(A\) corresponding to 2-faces containing \(\left\{u<w\right\}\) are supported on the columns corresponding to \(u\), and these rows are independent. Since the submatrix of the remaining rows and columns is precisely \(\psi_{K,m}^{2}\), we get:
    \[\mathrm{rank} A \geq 2 + \mathrm{rank} \psi_{K,m}^{2},\]
    and we are done.
\end{proof}

\begin{proposition}
\label{contraction in triangle preserves 2nd dimension}
    Let \(K\) be a triangulated surface.
    Let \(D\subset K\) be a subcomplex with \(\left|D\right|\) a disk, and \(\partial D\) a triangle. If \(K\rightarrow K'\) is a contraction of an edge internal to \(D\), then:
    \[\mathrm{Tail}_{\mathrm{lex}}\left(\Delta\left(K\right),\left( 1,4,5\right) \right)=\mathrm{Tail}_{\mathrm{lex}}\left(\Delta\left(K'\right),\left( 1,4,5\right) \right).\]
\end{proposition}
\begin{proof}
    Denote \(\partial D_{0}=\left\{ a,b,c\right\} \). If \(abc\) is a 2-face in \(K'\), then \(K\) is achieved by gluing a tetrahedron \(T\) to \(K'\) at \(abc\) along one of its faces, and then removing \(abc\), a.k.a. \emph{stacking}. By \ref{Shifting a union over a simplex}
    we have:
    \[\Delta\left(K'\coprod_{abc}T\right)_{2}=\Delta\left(K'\right)_{2}\cup\left\{ \left(1,2,m_{1,2}\right),\left(1,3,m_{1,3}\right),\left(2,3,m_{2,3}\right)\right\} ,\]
    where \(m_{i,j}=\min\left\{ m\in\mathbb{N}|\ m>j,\ \left(i,j,m\right)\notin\Delta\left(K'\right)\right\} \). Since \(K\) is attained by deleting \(abc\) from \(K'\coprod_{abc}T\), and this reduces the second betti number by one, we get by \ref{Shifting preserves Betti numbers}
    that:
    \[\Delta\left(K'\coprod_{abc}T\right)_{2}\backslash\Delta\left(K\right)_{2}=\left(2,3,m_{2,3}\right),\]
    which shows what we wanted. A similar argument shows that the same holds when \(abc\) is not a 2-face in \(K'\), using the fact that for any triangulation \(S\) of a 2-sphere, we have (see \cite{KalaiShifting}):
    \[\Delta\left(S\right)_{2}=\left\{ \sigma\in{\left[n\right] \choose 3}|\ \sigma\leq_{\mathrm{lex}}\left(1,3,n\right)\right\} \cup\left\{ \left(2,3,4\right)\right\} .\]
\end{proof}

\section{Critical Regions}
\label{critical regions section}

Our goal here is to define and prove important properties of critical regions (see Definition 4.5 onwards), which are a key tool to proving Theorem \ref{alg exists}. Their defining feature is that \ref{vert split preserves tails} holds with equality of sets for edges, roughly speaking. This feature is stated precisely in Proposition \ref{vert split preserves crit regions}.

\begin{definition}
    For a surface-embeddable complex \(D\) with \(n\) vertices, and \(1\leq k \leq n\), we shall denote by \(M_{D,k}\) the submatrix of \(\psi_{D,k}^{1}\) obtained by restricting the rows to those corresponding to the internal edges of \(D\), and restricting to columns to those corresponding to the internal vertices of \(D\).
\end{definition}
 
\begin{proposition}
\label{vert split preserves crit regions}
    Let \(D' \subseteq K'\) be surface-embeddable complexes, and let \(K'\rightarrow K\) be a contraction of an internal edge of \(D'\). Let \(D\subseteq K\) be the subcomplex obtained from \(D'\) by contracting the aforementioned internal edge. Assume also that \(D'\) has three more internal edges than does \(D\).\footnote{This is equivalent to assuming that the contraction does not create a diagonal.} For \(k\geq 3\), if the rows of \(M_{D,k}\) are linearly independent, then:
    \[\mathrm{Tail}_{\mathrm{lex}}\left(\Delta\left(K\right),\left\{ k+1,k+2\right\} \right)=\mathrm{Tail}_{\mathrm{lex}}\left(\Delta\left(K'\right),\left\{ k+1,k+2\right\} \right).\]

    This holds even when \(M_{D,k}\) is a \(0\times 0\) matrix, which happens precisely when \(D\) consists of a single 2-face.
\end{proposition}
\begin{proof}
    We shall show equality of sizes for all $k\ge 3$:

    \begin{equation}
        \label{eq:tails}\left|\mathrm{Tail}_{\mathrm{lex}}\left(\Delta\left(K\right),\left\{ k+1,k+2\right\} \right)\right|=\left|\mathrm{Tail}_{\mathrm{lex}}\left(\Delta\left(K'\right),\left\{ k+1,k+2\right\} \right)\right|.
    \end{equation}
   
To show it, it is enough to show that the rows of \(M_{D',k}\) are also linearly independent. Indeed, as the matrices of \(\psi_{K,k}^{1}\) and \(\psi_{K',k}^{1}\) have the forms:
    \[\begin{pmatrix}A & 0\\
\ast & M_{D,k}
\end{pmatrix},\begin{pmatrix}A & 0\\
\ast & M_{D',k}
\end{pmatrix}\] respectively,
    for some matrix \(A\), and the rows of each of \(M_{D,k}\) and \(M_{D',k}\) are assumed to be independent, and \(M_{D',k}\) has three more rows than \(M_{D',k}\), we conclude that the cokernels of \(\psi_{K,k}^{1}\) and \(\psi_{K',k}^{1}\) have the same dimension, which yields (\ref{eq:tails}). 
    
Next, since (\ref{eq:tails}) will also hold for all \(k'\geq k\) (as \(M_{D,k}\) is a submatrix of \(M_{D,k'}\) with the same number of rows), we will get the desired equality of sets, since for \(k+1\leq x<y\) we have:
    \[\left\{ x,y\right\} \in\mathrm{Tail}_{\mathrm{lex}}\left(\Delta\left(K\right),\left\{ k+1,k+2\right\} \right)\]
    if and only if:
    \[y-x\leq\left|\mathrm{Tail}_{\mathrm{lex}}\left(\Delta\left(K\right),\left\{ x,x+1\right\} \right)\right|-\left|\mathrm{Tail}_{\mathrm{lex}}\left(\Delta\left(K\right),\left\{ x+1,x+2\right\} \right)\right|.\]

    In the case that \(D\) is a single 2-face, then \(M_{D',k}\) is a generic \(3\times k\) matrix, so its rows are linearly independent, as desired.

    In the general case, we mimic the proof of \cite[Lemma 3.3.1]{nevo2007algebraic}: Let \(\left\{ u,v\right\} \) be the edge of \(D'\) we contract (contracting \(v\) to \(u\)), with \(v\) internal in \(D'\). Without loss of generality \(u > v\), and all other vertices are smaller than \(v\).

    Now we replace all \(x_{iv}\) in \(M_{D',k}\) with \(x_{iu}\) to get a new matrix \(B\). 
     
     It is enough to show that the rows of \(B\) are independent to know that the same is true for \(M_{D',k}\).

    Assume some linear combination of the rows of \(B\) equals zero. Let \(\overline{B}\) be the matrix obtained from \(B\) by adding the columns of \(u\) to the corresponding columns of \(v\), and deleting the columns of \(u\). Then, a linear combination of the rows of \(\overline{B}\) with the same coefficients also equals zero. Note that \(\overline{B}\) can be obtained from the matrix of \(M_{D,k}\) by adding a zero row (for the edge \(\left\{ u,v\right\} \)), and doubling the rows of the edges \(\left\{ u,w\right\} \) where \(w\) neighbors \(u\) and \(v\) in \(D'\) (or adding a row of zeroes for \(\left\{ u,w\right\} \) if \(\left\{ u,w\right\} \) is not an internal edge in \(D\)). Since the rows of \(M_{D,k}\) are independent, all the coeffecients must be zero save for the doubled rows and the zero rows, and further, the coefficients of the pairs of doubled rows must differ only by sign. The submatrix of \(B\) consisting of the mentioned rows, and the columns of \(v\), is a generic \(3\times k\) matrix, and therefore its rows are independent. We conclude that all coefficients are zero, as desired.
\end{proof}

We need the following combinatorial-topological definition in order to define critical regions:

\begin{definition}
    For a simplicial complex \(K\) and a set \(C\subseteq K_{d}\) (for some \(d\geq 0\)), we say that \(K\) is \emph{\(C\)-connected} if for any two different faces \(\sigma,\sigma' \in K_{d+1}\), there exists some sequence of $(d+1)$-faces in \(K_{d+1}\):
    \[\sigma=\sigma_{0},\sigma_{1},\ldots,\sigma_{k-1},\sigma_{k}=\sigma'\]
    such that \(\sigma_{i-1}\cap\sigma_{i}\in C\) for all \(1\leq i \leq k\). In the special case where \(K\) is surface-embeddable and \(C\) is the collection of internal edges in \(K\), we shall say that \(K\) is \emph{internally 1-connected} if it is \(C\)-connected.

Note that if  \(K\) is a disc then it is internally 1-connected if and only if it has no diagonal.

    We shall say that a surface-embeddable complex \(K\) is \emph{internally connected} if the induced subcomplex attained by restricting \(K\) to its internal vertices is connected.
\end{definition}

\begin{lemma}
    An internally 1-connected surface-embeddable complex is internally connected.
\end{lemma}
\begin{proof}
    Let \(u,v\) be distinct internal vertices in an internally 1-connected surface-embeddable complex  \(K\). Then there exists some sequence \(\sigma_{0},\ldots,\sigma_{n}\) of 2-faces in \(K\) such that each pair of consecutive elements share an edge that is internal, with \(u \in\sigma_{0}\) and \(v \in\sigma_{n}\).
    
    Clearly, if \(v\) lies in the link of \(u\), we are done. Otherwise, let \(1 \leq m\leq n\) be the smallest index where \(u\notin \sigma_{m}\). Then \(\sigma_{m-1}\cap\sigma_{m}\) contains some internal vertex \(w\), and \(\left\{u,w\right\}\subseteq \sigma_{m-1}\) is an edge in \(K\). The proof proceeds by induction, by repeating the above for \(w\) and \(v\), and the shorter sequence \(\sigma_{m},\ldots,\sigma_{n}\).
\end{proof}

\begin{definition}
    Let \(C\) be an internally-1-connected surface-embeddable complex. We call \(C\) a \emph{critical region} if the rows of \(M_{C,4}\) are linearly independent. If \(M_{C,4}\) is a square matrix, we say that \(C\) is \emph{irreducible} (as a critical region), otherwise \(C\) is \emph{reducible}.

    If \(C\) is contained in some triangulated surface \(K\), we say that \(C\) is a critical region of \(K\). A vertex split \(C \rightarrow C'\) such that \(M_{C',4}\) has three more rows than \(M_{C,4}\) is called a \emph{critical split} of \(K\) (or of \(C\)).

    If a triangulated surface \(K\) contains no reducible critical regions, we say that \(K\) is \emph{critically irreducible}.
\end{definition}
\begin{remark}
\label{crit region remarks}
    a) If a critical region \(C\) is a disk with \(n\) internal vertices and $b$ boundary vertices, we compute a lower bound on \(n\): Using the Euler characteristic and the fact that \(C\) has no diagonals, one can compute that \(C\) has \(3n + b -3\) internal edges. Since the rows of \(M_{C,4}\) are independent, there are at least as many columns as rows, which yields the inequality:

    \[4n \geq 3n + b -3.\]

    Simplifying, we get \(n \geq b-3\), with equality if and only if \(C\) is irreducible.
    
    b) A contraction of an internal edge \(C' \rightarrow C\) removes three internal edges as long as it does not create a diagonal. Therefore, the split \(C \rightarrow C'\) is critical if and only if it does not destroy a diagonal.
\end{remark}

We now prove the following crucial proposition, which gives us both a combinatorial characterization of critical regions with small boundary, and the existence of edge contractions in them that make the critical region irreducible; the two are proved in tandem.  
\begin{proposition}
\label{characterization of crit disks}
    (i) For \(3\leq b\leq 6\), an internally 1-connected triangulated disk $D$ with \(b\) boundary vertices is a critical region if and only if it contains no internal vertex adjacent to more than four boundary vertices of $D$, and no pair of adjacent internal vertices each adjacent to four boundary vertices of $D$.
    
    (ii) For a pinched disk obtained by gluing two opposite vertices of a hexagonal triangulated disk $D$, and for a Möbius strip obtained by gluing two opposite edges of $D$, if $D$ is a critical region, then the resulting glued disk is also a critical region.

    (iii) For \(3\leq b\leq 6\), if a disc $D$ with $b$ boundary vertices is a critical region, then there exists a sequence $D=D(0),\ldots, D(t)$ such that for every $1\le i\le t$, the disc $D(i)$ is obtained from $D(i-1)$ by a contraction of an internal edge such that $D(i)$ is also a critical region (call such edge contraction \emph{admissible}) and $D(t)$ is an irreducible critical region.

    The same holds when replacing the disc $D$ by a pinched disks or a Möbius strip as in (ii).
    
\end{proposition}
\begin{definition}
    A critical region characterized in Proposition \ref{characterization of crit disks} will be called a \emph{combinatorial} critical region. The \emph{type} of a combinatorial critical region \(C\) is the set of all combinatorial critical regions that have the same geometric realization as \(C\) and the same number of boundary vertices.
\end{definition}

In order to prove \ref{characterization of crit disks}, we will need some preparatory lemmas:

\begin{lemma}
\label{crit reg necessary condition}
    Let \(C\) be surface-embeddable, and \(V\) some collection of internal vertices of \(C\). Let \(B\) be the set of boundary vertices of \(C\), and define the following collection of internal edges:
    \[E_{V}^{B}:=\left\{ e\in C_{1}|\ e\cap V\neq\emptyset,e\subseteq B\cup V\right\} .\]
    If \(C\) is a critical region, then \(\left|E_{V}^{B}\right| \leq 4\cdot\left| V\right|\).
\end{lemma}
\begin{proof}
    The rows of \(M_{C,4}\) corresponding to the edges in \(E_{V}^{B}\) are supported on the columns corresponding to vertices in \(V\), and are linearly independent.
\end{proof}

The following is a stronger version of~\cite[Lemma 6]{Whiteley_2005}, and the proof is essentially the same:
\begin{lemma}
\label{contractible edge exists in disk}
    Let \(D\) be an internally connected disk with more than one internal vertex. Then there is some contractible edge between two internal vertices in \(D\).
\end{lemma}
\begin{proof}
    There exists some edge \(e_{1}\) between two internal vertices of \(D\) by internal connectedness. For \(i \geq 1\), if  \(e_{i}\) is not part of a missing triangle, we are done, as it is contractible. 
    Otherwise, it borders a smaller triangular disk \(D_{i}\) in \(D_{i-1}\) (where \(D_{0}:=D\)), and there exists some edge \(e_{i+1}\) between two internal vertices of \(D\), such that \(e_{i+1}\) is internal to \(D_{i}\).

    By continuing this process, we get a strictly descending sequence of disks \(D_{0}\supset D_{1}\supset\cdots\) which must terminate at some \(D_{n}\), since \(D\) is finite, which means \(e_{n+1}\) is contractible.
\end{proof}
\begin{proof}[Proof of Proposition \ref{characterization of crit disks}]
    Outline of the proof: in Step 1 we show the ``only if" direction of (i). In Step 2 we show that the ``if" direction of (i), as well as (ii), hold in the irreducible cases. In Step 3, based on the results in Step 2 as a base case, 
    we prove the ``if" part of (i)
    in the reducible cases as well, inductively, in tandem with (iii) for discs. We leave the rest of the proof, namely part (ii) in the reducible case, as well as 
    part (iii) for M\"obius strips and pinched disks
    to the Appendix, as they are not necessary for the torus.

    
    \underline{Step 1:} The ``only if" direction in (i) is straightforward: By Lemma~\ref{crit reg necessary condition}, a critical region cannot contain an internal vertex adjacent to more than four boundary vertices, or a pair of adjacent internal vertices each adjacent to four boundary vertices.

    \underline{Step 2:} We now show the ``if" direction in (i) for the case when \(D\) has exactly \(b-3\) internal vertices; in this case the resulting critical regions would be irreducible, as noted above in Remarks \ref{crit region remarks}.
    In this case, the proof amounts to showing that \(M_{C,4}\) is nonsingular when \(C\) is a triangulated disk, pinched disk or Möbius strip as described.

    When $b=3$, \(C\) is a triangle and the claim holds trivially, as \(M_{C,4}\) has no rows.
    
    For $b=4$, \(C\) is a quadrilateral disk with one internal vertex, which is adjacent to all boundary vertices. In this case, \(M_{C,4}\) is a \(4\times 4\) generic matrix.
    
    \begin{figure}
        \centering
        \includegraphics{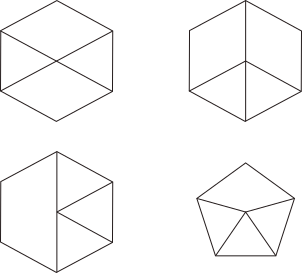}
        \caption{A vertex with four boundary neighbors subdivides a pentagon or a hexagon into smaller critical regions.}
        \label{fig:vertex 4 neighbors}
    \end{figure}
    
    For the other cases ($b=5,6$), first assume that \(v\) is an internal vertex of \(C\) that neighbors four boundary vertices. In this case, in the rows of \(M_{C,4}\) corresponding to the edges between \(v\) and the boundary, there are zeroes in all columns but those corresponding to \(v\), and the \(4\times 4\) submatrix of said rows and columns is generic. Restricting to the remaining rows and columns of \(M_{C,4}\), we get a matrix corresponding to a quadrilateral irreducible critical region when $b=5$, so we are done by the argument above. When \(C\) is hexagonal disk ($b=6$) (or a pinched disk or Möbius strip glued from a hexagonal disk), this submatrix corresponds to either a pentagonal irreducible critical region, or to the union of two quadrilateral irreducible critical regions, as illustrated in Figure \ref{fig:vertex 4 neighbors}, and again, we are done by the argument above.~\footnote{See Lemma~\ref{lem:critical subregions give critical region} for a generalization of this argument.}
 

    \begin{figure}
        \centering
        \includegraphics{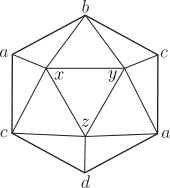}
        \caption{Triangulation of a hexagon with each internal vertex neighboring three boundary vertices, glued to form a Möbius strip \(S\).}
        \label{fig:min mobius}
    \end{figure}
    The only case we have not dealt with is when no such \(v\) exists, and this is only possible in a hexagon with three inner vertices, each of which neighbors exactly three consecutive boundary vertices, as depicted in Figure \ref{fig:min mobius}.
    Indeed, if each internal vertex has at most three neighbors in the boundary then, if $b=5$ we get at most $7=3+3+1$ edges while we need 8 of them to triangulate $D$, and if $b=6$ then we have the 12 interior edges needed to triangulate $D$ if and only if the three internal vertices are pairwise adjacent and each of them neighbors exactly three boundary vertices. Note that the three vertices in the boundary neighboring the same internal vertex must be consecutive, as otherwise the corresponding three edges would create two discs with non-triangle boundary in $D$, thus the other two interior vertices would be one in each of these discs, but then the edge between them will cross one of the aforementioned edges, a contradiction.   
     
    In this last case we calculate the rank of the matrix by hand. In fact, it is enough to do this for the Möbius strip \(S\), since the Möbius strip is attained by gluing a hexagonal disk or a pinched disk to itself, which means that the matrix of the Möbius strip is achieved by identifying some generic entries in the matrix of the disk (or the matrix of the pinched disk). The matrix in question is:
     

\[M_{S,4}=\begin{pmatrix}\overset{x}{\overbrace{\begin{array}{cccc}
y_{1} & y_{2} & y_{3} & y_{4}\\[0.5em]
\underbrace{z_{1}} & z_{2} & z_{3} & z_{4}\\[0.5em]
0 & 0 & 0 & 0\\[0.5em]
-a_{1} & -a_{2} & -a_{3} & -a_{4}\\[0.5em]
-b_{1} & -b_{2} & -b_{3} & -b_{4}\\[0.5em]
-c_{1} & -c_{2} & -c_{3} & -c_{4}\\[0.5em]
0 & 0 & 0 & 0\\[0.5em]
0 & 0 & 0 & 0\\[0.5em]
0 & 0 & 0 & 0\\[0.5em]
0 & 0 & 0 & 0\\[0.5em]
0 & 0 & 0 & 0\\[0.5em]
0 & 0 & 0 & 0
\end{array}}}&\overset{y}{\overbrace{\begin{array}{cccc}
\underbrace{-x_{1}} & -x_{2} & -x_{3} & -x_{4}\\[0.5em]
0 & 0 & 0 & 0\\[0.5em]
z_{1} & z_{2} & z_{3} & z_{4}\\[0.5em]
0 & 0 & 0 & 0\\[0.5em]
0 & 0 & 0 & 0\\[0.5em]
0 & 0 & 0 & 0\\[0.5em]
-b_{1} & -b_{2} & -b_{3} & -b_{4}\\[0.5em]
-c_{1} & -c_{2} & -c_{3} & -c_{4}\\[0.5em]
-a_{1} & -a_{2} & -a_{3} & -a_{4}\\[0.5em]
0 & 0 & 0 & 0\\[0.5em]
0 & 0 & 0 & 0\\[0.5em]
0 & 0 & 0 & 0
\end{array}}}&\overset{z}{\overbrace{\begin{array}{cccc}
0 & 0 & 0 & 0\\[0.5em]
-x_{1} & -x_{2} & -x_{3} & -x_{4}\\[0.5em]
-y_{1} & \underbrace{-y_{2}} & -y_{3} & -y_{4}\\[0.5em]
0 & 0 & 0 & 0\\[0.5em]
0 & 0 & 0 & 0\\[0.5em]
0 & 0 & 0 & 0\\[0.5em]
0 & 0 & 0 & 0\\[0.5em]
0 & 0 & 0 & 0\\[0.5em]
0 & 0 & 0 & 0\\[0.5em]
-a_{1} & -a_{2} & -a_{3} & -a_{4}\\[0.5em]
\underbrace{-d_{1}} & -d_{2} & -d_{3} & -d_{4}\\[0.5em]
-c_{1} & -c_{2} & -c_{3} & -c_{4}
\end{array}}}
  
\end{pmatrix}\begin{aligned}
\begin{matrix}
  \overmat{}{\\[0.37em]
  \partialphantom xy}  \\[0.3em]
  \partialphantom xz  \\[0.3em]
  \partialphantom yz  \\[0.2em]
  \partialphantom ax  \\[0.38em]
  \partialphantom bx  \\[0.385em]
  \partialphantom cx  \\[0.385em]
  \partialphantom by  \\[0.385em]
  \partialphantom cy  \\[0.385em]
  \partialphantom ay  \\[0.385em]
  \partialphantom az  \\[0.385em]
  \partialphantom dz  \\[0.385em]
  \partialphantom cz  \\[0.385em]
  \end{matrix}
\end{aligned}.\]

    In particular, the coefficient of \(d_{1}x_{1}z_{1}y_{2}\) in the determinant of \(M_{S,4}\) is:

    \[\pm\det\begin{pmatrix}-a_{2} & -a_{3} & -a_{4}\\
-b_{2} & -b_{3} & -b_{4}\\
-c_{2} & -c_{3} & -c_{4}\\
 &  &  & -b_{2} & -b_{3} & -b_{4}\\
 &  &  & -c_{2} & -c_{3} & -c_{4}\\
 &  &  & -a_{2} & -a_{3} & -a_{4}\\
 &  &  &  &  &  & -a_{3} & -a_{4}\\
 &  &  &  &  &  & -c_{3} & -c_{4}
\end{pmatrix}\neq0.\]

and therefore \(\det (M_{S,4})\) is nonzero, as required.

    \underline{Step 3:} Now we prove the ``if" part of (i), and (iii), when \(D\) has more than \(b - 3\) internal edges.
    This is done by double induction, on $b=3,4,5,6$ and $n\ge b-3$, the number of internal vertices, where the base case $n=b-3$ was proved in Step 2. 
    
In words,  by \ref{vert split preserves crit regions} and induction on the number of internal 
    vertices, we will show that we can sequentially contract contractible edges in \(|D|\) until we reach an irreducible critical region, such that every intermediate internally 1-connected triangulated disk also contains
no internal vertex adjacent to more than four boundary vertices of $|D|$, and
no pair of adjacent internal vertices each adjacent to four boundary vertices of $|D|$.

    For \(b=3,4\), by Lemma \ref{contractible edge exists in disk} we can always contract to the appropriate irreducible critical disk.
    
    For \(b=5,6\), if a vertex $v$ in \(D\) neighbors four boundary vertices, it subdivides $|D|$ into four smaller discs, each has less than $b$ boundary vertices, and we will reduce the problem to contracting interior edges in these smaller disks (see Figure \ref{fig:vertex 4 neighbors}): as these edges are clearly also interior edges in $D$ it will show (iii), and, by counting, once each small disc becomes an irreducible critical region then so becomes $D$, which will show the ``if" part of (i) by induction. 
    
    Now, for the reduction, we only need to show that these smaller discs are \emph{critical} regions as well. For this we verify that each smaller disc is both (a) internally $1$-connected, namely has no diagonal, and (b) contains non of the two forbidden subgraphs defining a combinatorial critical region. Indeed, for (a) note that a diagonal in one of the smaller discs is either a diagonal in $D$, which is impossible as $D$ is internally $1$-connected, or connects $v$ to the boundary of $D$, making $v$ a neighbor of more than four boundary vertices in $D$, a contradiction. For (b) note that the only possible smaller disc with more than four boundary vertices occurs in the case of Figure \ref{fig:vertex 4 neighbors} bottom left, where there there is a unique smaller disc with five boundary vertices, but if it has an interior vertex $u$ that neighbors all five boundary vertices, in particular $v$, then $D$ contains the forbidden subgraph with internal edge $vu$ and each of $v$ and $u$ neighbors four of the boundary vertices in $D$, a contradiction.

    

    \begin{figure}
        \centering
        \includegraphics{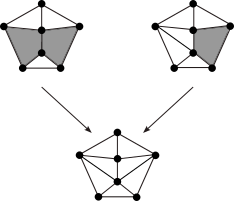}
        \caption{How to contract in a pentagon when two neighboring internal vertices collectively neighbor 5 boundary vertices. An arrow from A to B means that we can contract in the gray areas in A in order to get B. }
        \label{fig:pentalg}
    \end{figure}
    
    \begin{figure}
        \centering
        \includegraphics[scale=0.7]{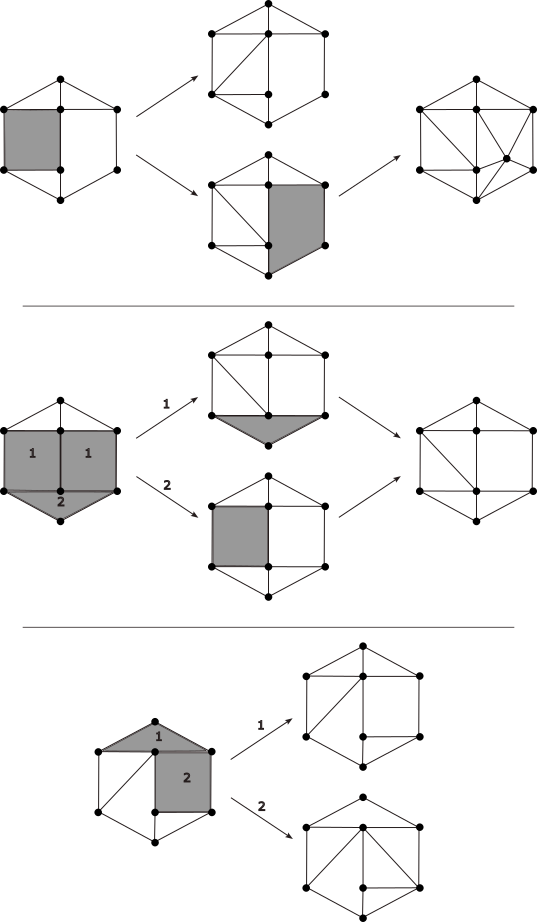}
        \caption{How to contract in a hexagon when two neighboring internal vertices collectively neighbor five or six boundary vertices. Using the 
        notation for the partition of interior vertices $A, B, C$,
        the upper part corresponds to the cases when \(w\) is adjacent to a member of \(A\) and a member of \(B\), the middle part is when \(w\) is adjacent to two members of \(B\) (or when \(C\) is empty), and the bottom part is when \(w\) is adjacent to two members of \(A\).}
        \label{fig:hexalg}
    \end{figure}
    Thus, we may assume now that \(D\) has no internal vertex with four boundary neighbors. Let \(e\) be some contractible edge in \(D\) between two internal vertices $u$ and $v$, which exists by  Lemma \ref{contractible edge exists in disk}. If contracting \(e\) yields a vertex with more than four boundary neighbors, then we are in one of the situations given in Figures \ref{fig:pentalg} and \ref{fig:hexalg}. In these figures, we present all possible configurations, as we shall prove,
    up to symmetries of the boundary 5- or 6-gon, with two neighboring internal vertices \(u\) and \(v\), where \(u\) neighbors three boundary vertices, and \(v\) neighbors two or three additional boundary vertices.
    To see that indeed we exhaust all possibilities, note that we have a partition \(A\cup B\cup C\) of the boundary vertices, where \(A\) contains boundary vertices neighboring \(u\), \(B\) contains boundary vertices neighboring \(v\) but not \(u\), and \(C\) contains the rest of the boundary vertices. Then,
    $3=|A|\ge |B|\ge 2$ and $C$ is either a singleton (when $b=6$ and $|B|=2$) or empty (otherwise). In a hexagon, if \(C\) contains a vertex \(w\), we then have three possibilities for the position of \(w\): it can either neighbor two elements of \(A\), or two elements of \(B\), or one from $A$ and one from $B$. Otherwise \(C\) is empty, and 
    $A$ and $B$ partition the vertex set of the boundary cycle  into two intervals.
    Note that the case
    when $b=6$ and $|B|=3$ appears in Figure~\ref{fig:hexalg} in the middle of the middle row, and all other cases
    appear in Figures \ref{fig:pentalg} and \ref{fig:hexalg} as the discs $D$ admitting an outward going
    arrow. In these figures, the interior upper vertex is $u$ and the interior lower vertex is $v$. 
     
    In each case, the edges containing $v$ or $u$ (or both) subdivide \(D\) into smaller disks, and we can contract admissible interior edges inside these subregions (marked in gray) until (i) either we obtain an internal vertex with four boundary neighbors in $|D|$, 
    a case we have already solved,  
    (ii) or 
    we have three internal vertices each neighboring three boundary vertices, as depicted in the rightmost case in the upper part of Figure \ref{fig:hexalg}, thus any remaining internal edges lie inside triangular critical disks in $D$, and the problem then reduces to contracting interior edges in said triangular disks. See the discs with no outgoing arrows in in Figures \ref{fig:pentalg} and \ref{fig:hexalg} illustrating situations (i) and (ii). This completes the proof of parts (i) and (iii) of the proposition for discs.

    For the proof of part (ii) 
    of the proposition in the reducible case and of part (iii) for $D$ a pinched disc or a M\"obius strip, see the Appendix. This case is only needed for the Klein bottle case in Theorems \ref{alg exists} and \ref{possible shiftings}.
\end{proof}

The following simple lemma generalizes an argument used in the proof of~\ref{characterization of crit disks}, and looks useful beyond our needs here, so we highlight it:
\begin{lemma}
\label{lem:critical subregions give critical region}
Let $D$ be an internally-1-connected surface-embeddable complex, and $C$ a graph consisting of some interior edges in $D$ that subdivides $D$ into subregions $D(1),\ldots,D(k)$. Assume that each $D(i)$ is a critical region, $1\le i\le k$, and that the the rows of $M_{D,4}$ corresponding to edges in $C$ are linearly independent. Then $D$ is a critical region.  
\end{lemma}
\begin{proof}
For a subcomplex $X$ of $D$, denote by $E'(X)$ (resp. $V'(X)$)  the edges (resp. vertices) of $X$ interior in $D$.
Then $E'(C), E'(D(1)),\ldots, E'(D(K))$ partition $E'(D)$ and 
$V'(C), V'(D(1)),\ldots, V'(D(K))$ partition $V'(D)$. Write $M_{D,4}$ according to this ordered partition.
Each of the block submatrices of $M_{D,4}$ corresponding to the rows $E'(D(i))$ and columns of $V'(D(i))$ equals $M_{D(i),4}$ and thus has independent rows; same holds for the rows of $E'(C)$ by assumption.
Now, each block submatrix with rows corresponding to $E'(D(i))$ and columns corresponding to $V'(D(j))$ for $i\neq j$ is a zero matrix, and the same holds for the submatrices corresponding to $E'(C)$ and $V'(D(i))$ for $1\le i\le k$.
Thus, $M_{D,4}$ is a lower block diagonal matrix where the rows in each diagonal block are linearly independent, hence the rows of $M_{D,4}$ are linearly independent.
\end{proof}

\section{An Algorithm for Shifting a Torus}
\label{algorithm section}

We start with an important corollary of the commutativity lemma. 
\begin{proposition}
\label{commutativity for non crit splits}
    Assume \(K\) and \(K'\) are two triangulations of the same surface, where \(K'\) is obtained by a sequence of vertex splits from \(K\). Assume further that \(K\) is critically irreducible, and no boundary vertex of a maximal critical regions in \(K\) (w.r.t. inclusion) is internal to another critical region.

    If all critical regions in \(K\) are combinatorial, and all combinatorial critical regions in \(K'\) that share a type with a critical region in \(K\) are irreducible, then there exists a sequence of vertex splits from \(K\) to \(K'\) where the first split is not a critical split.
\end{proposition}
\begin{proof}
    Let us write the sequence of vertex splits:
    
    \[K=K_{0}\rightarrow K_{1}\rightarrow\cdots\rightarrow K_{n-1}\rightarrow K_{n}=K'.\]

    By the hypothesis both \(K\) and \(K'\) are prime, therefore by \ref{prime triangulation survival} we can assume no split occurs inside a triangular disk. If the only critical regions in \(K\) are triangular disks, we are done. 
    
    Otherwise, assume the first split occurs inside some maximal critical region \(C\) in \(K\). As \(C\) is combinatorial, and \(K'\) contains no reducible combinatorial critical region of the same type as \(C\), there must be some minimal \(1<t \leq n\) such that the split from \(K_{t-1}\) destroys the induced graph on the boundary vertices of \(C\) (either destroys \(\partial C\) or destroys a diagonal in \(C\)). Then by \ref{commutativity}, we can replace the sequence with another sequence of vertex splits where said graph is destroyed on the first split, as there is no simple path \((u,v,w)\) in this graph where \(uvw\) is a 2-face, since \(C\) is internally-1-connected (and since \(C\) is not a triangle, as we assumed that \(K_{1}\) is prime). Since the first split of this sequence occurs on a boundary vertex of \(C\), and by the assumption no boundary vertex of \(C\) is internal to another critical region, we are done.
\end{proof}

\begin{theorem}
\label{shifting of a large torus}
    If \(T\) is a triangulated torus with \(n\) vertices, then for \(n \geq 11\), if \(T\) contains no reducible critical disks with six boundary vertices or less, then:
    \begin{equation}
    \label{eq:edges of shifted torus}
        \Delta\left(T\right)_{1}=\left\{ \sigma\in{\left[n\right] \choose 2}|\ \sigma\leq_{\mathrm{lex}}\left(4,10\right)\right\}.
    \end{equation}
    Similarly, for \(n\geq 8\), if \(T\) is prime, then:
    \begin{equation}
    \label{eq: two-faces of shifted torus}
        \Delta\left(T\right)_{2}=\left\{ \sigma\in{\left[n\right] \choose 3}|\ \sigma\leq_{\mathrm{lex}}\left(1,4,8\right)\right\} \cup\{\left( 2,3,4\right)\}.
    \end{equation}
\end{theorem}
\begin{proof}
    


Assume first that $n=11$ for $T$ as in the theorem.
    Using a computer, we calculated that 
    in this case (\ref{eq:edges of shifted torus}) holds for $T$. 
    Furthermore, we verified by computer calculation that for \emph{every} critically irreducible triangulation $T'$ of the torus on at most 10 vertices such that (\ref{eq:edges of shifted torus}) fails for $T'$, the following three properties hold: (a) all critical regions in \(T'\) are disks with at most six boundary vertices, (b) no boundary vertex of a maximal critical regions in \(T'\) (w.r.t. inclusion) is internal to another critical region in $T'$
    and (c) if $T''$ is obtained from $T'$ by a non-critical split and (\ref{eq:edges of shifted torus}) fails for $T''$ then (b) holds for $T''$ as well.~\footnote{Note that (c) is implied by (b) and the calculation mentioned for $n=11$. We phrased (c) for stressing the inductive argument via vertex splits.}

    
    Now assume that \(T\) has \(n > 11\) vertices, and contains no reducible critical disks having six or fewer boundary vertices. There is some sequence of edge contractions \(T\rightarrow\cdots\rightarrow T_{0}\) with \(T_{0}\) an irreducible triangulation.

    By the calculations we mentioned, including properties (a,b,c) above, we can repeatedly apply \ref{commutativity for non crit splits} to modify the intermediate triangulations in the above sequence of edge contractions, and (reversing the sequence) obtain vertex splits  \(T_0\rightarrow\cdots T'\rightarrow\cdots \rightarrow T\) where
    \(T'\) has \(11\) vertices and (\ref{eq:edges of shifted torus}) holds for \(T'\).

    Now, \(\mathrm{Tail}_{\mathrm{lex}}\left(\Delta\left(T'\right),\left\{ 5,6\right\} \right)=\emptyset\), and therefore by \ref{vert split preserves tails}: \[\mathrm{Tail}_{\mathrm{lex}}\left(\Delta\left(T\right),\left( 5,6\right)\right)=\emptyset.\] We claim that this completely determines the edges of \(\Delta\left(T\right)\). Indeed, by \ref{volume rigidity of torus kb and rp2} we have \(\left(3,n\right)\in\Delta\left(T\right)\), and by shiftedness we get \(\sigma\in\Delta\left(T\right)\) for all \(\sigma\leq_{\mathrm{lex}}\left(3,n\right)\). This gives us \(3n - 6\) edges we know are in \(\Delta\left(T\right)\), leaving us with six more edges to find.\footnote{By the Euler characteristic \(T\) has \(3n\) edges and \(2n\) 2-faces.} As these edges must be chosen from \(\left\{ \left(4,k\right)|\ 4<k\leq n\right\} \), by shiftedness we get (\ref{eq:edges of shifted torus}).

    For (\ref{eq: two-faces of shifted torus}) first note that \(\left(2,3,4\right)\in\Delta\left(T\right)\) to account for nontrivial second homology (by \ref{Shifting preserves Betti numbers}). We calculated that (\ref{eq: two-faces of shifted torus}) holds for all irreducible triangulations with eight or more vertices, as well as for all 
    reducible prime triangulations with eight vertices (there are precisely two of the latter, see \cite{LavrenchenkoTorus}). We then get the desired result using a similar argument as before: if \(T\) is reducible and has more than eight vertices, then  
     there is a sequence of edge contractions to either an irreducible triangulation on at least $8$ vertices, or
     to a prime triangulation on $8$ vertices.
     Then in particular \(T\) contracts to one of the above cases in which we calculated that (\ref{eq: two-faces of shifted torus}) holds, and therefore (\ref{eq: two-faces of shifted torus}) must hold for \(T\) as well by \ref{vert split preserves tails}.    
\end{proof}

\begin{remark}
\begin{itemize}
    \item There is only one triangulation of a torus with less than eight vertices, the unique irreducible triangulation with seven vertices.
    \item The calculations were not done deterministically, but rather by using a randomized matrix (in place of a generic one). Since doing so can only create linear dependencies that would not be there had we used a generic matrix, this means that if the calculated shifting for some complex gives a lexicographic prefix, then it must be correct. As stated in the proof, under the conditions of Theorem~\ref{shifting of a large torus} we calculated that for a sufficiently large triangulation (\(\geq 11\) vertices for shifting edges, \(\geq 8\) vertices for shifting 2-faces) the shifting is a lexicographic prefix\footnote{For the 2-faces we get a prefix discounting the face \((2,3,4)\), which we know is in the shifting by \ref{Shifting preserves Betti numbers}.}, therefore the calculation is sufficient for our needs.
    
    In fact, the calculations gave lexicographic prefixes for almost all the cases we calculated, and in particular for all irreducible triangulations. Moreover, the few cases where we did not get a prefix (there are less than 40 of these) 
    were cases with 10 vertices, in which the edge \((5,6)\) was found in the shifting (instead of \((4,10)\)). Therefore, in these cases there are only two possible shiftings, those being \(\mathcal{T}_3\) and \(\mathcal{T}_4\) of Theorem \ref{possible shiftings}, and both are already known to be a shifting of some other triangulation.
\end{itemize}
    
\end{remark}

\begin{proof}[Proof of Theorem \ref{possible shiftings} (i)]

    Suppose \(T\) is some triangulation of the torus with \(n\) vertices. Note that if \(\mathcal{T}_{i} \subseteq \Delta(T)\) for some \(1\leq i \leq 4\), then by counting faces (\(T\) has \(2n\) 2-faces and \(3n\) edges by the euler characteristic, and therefore so does \(\Delta(T)\)), \(\mathcal{T}_{i}\) must equal the collection of maximal faces of $\Delta(T)$ w.r.t. \(\leq_{p}\). Thus we simply need to show that \(\mathcal{T}_{i} \subseteq \Delta(T)\) for some \(1\leq i \leq 4\).
    
    If \(T\) is not prime, then we can contract inside triangular critical disks until reaching a prime triangulation \(T'\), and then (by \ref{contraction in triangle preserves 2nd dimension} and \ref{vert split preserves crit regions}):
    \begin{equation}
        \label{eqn: t edges}\Delta\left(T\right)_{1}=\Delta\left(T'\right)_{1}\cup\left\{ \sigma\in{\left[n\right] \choose 2}|\ \sigma\leq_{\mathrm{lex}}\left(3,n\right)\right\}\tag{$\star$},
    \end{equation}
    \[ \]
    \[\Delta\left(T\right)_{2}=\Delta\left(T'\right)_{2}\cup\left\{ \sigma\in{\left[n\right] \choose 3}|\ \sigma\leq_{\mathrm{lex}}\left(1,3,n\right)\right\} .\]
    Therefore it is enough to show the claim for \(T\) prime. Then, if \(n=7\), then \(T\) is irreducible, and calculation gives us \(\Delta\left(T\right)\supset\mathcal{T}_{1}\). Otherwise (if \(n\geq 8\)), by \ref{shifting of a large torus} we have \(\left(1,5,6\right)\notin\Delta\left(T\right)\) and thus it is enough to determine the edges of the shifting.
    
    As before, we reduce to the case where \(T\) is critically irreducible and \(\left(1,5,6\right)\notin\Delta\left(T\right)\), as otherwise we can contract edges in critical regions until reaching a critically irreducible triangulation \(T'\), and then by \ref{vert split preserves crit regions}, (\ref{eqn: t edges}) holds for such $T$ and $T'$. Note in particular that if \(\left(1,5,6\right)\notin\Delta\left(T\right)\), then \(\left(6,7\right)\notin\Delta\left(T\right)\) as well, otherwise we would have three edges \(\left(5,6\right),\left(5,7\right),\left(6,7\right)\) contributing to the 1st homology by \ref{Shifting preserves Betti numbers}, contradicting the fact that the 1st Betti number of the torus is \(2\). Thus under these assumptions \(n \geq 8\).
    
    Now, either \(n\geq 11\) and then by \ref{shifting of a large torus} we get \(\Delta\left(T\right)\supset\mathcal{T}_{4}\), or else \(8\leq n \leq 10\) and we are in one of the cases we calculated, and in these cases \(\Delta\left(T\right)\) contains one of \(\mathcal{T}_{2},\mathcal{T}_{3},\mathcal{T}_{4}\).
\end{proof}

\begin{proof}[Proof of Theorem \ref{alg exists} for the torus]
    Let \(T\) be a triangulated torus with \(n\) vertices. We first present the algorithm:\newline

    \underline{Step 0} Set \(T' := T\).

    \underline{Step 1} Find a reducible critical disk \(D\) in \(T'\) with at most six boundary vertices, of smallest boundary.

    \underline{Step 2} Contract edges in \(D\) until it becomes an irreducible critical disk (this is possible by 
     \ref{characterization of crit disks}(iii)), 
    and replace \(T'\) with the contracted triangulation.

    \underline{Step 3} Repeat steps 1-2 until no reducible disks are found in step 1 (thereby reaching a triangulation \(T'\) of the torus such that \(T'\) contains no reducible critical disk with six or less boundary vertices).

    \underline{Step 4} Calculate \(\Delta\left(T'\right)_{1}\) (using Theorem \ref{shifting of a large torus}).

    \underline{Step 5} Calculate \(\Delta\left(T\right)_{1}\) based on \(\Delta\left(T'\right)_{1}\).

    \underline{Step 6} Calculate \(\Delta\left(T\right)_{2}\) based on the first prime triangulation reached in the process of contracting edges in steps 1-2.
    \newline

    First, we explain steps 4-6. In step 4, if \(T'\) has at least 11 vertices, then Theorem \ref{shifting of a large torus} gives us \(\Delta\left(T'\right)_{1}=\left\{ \sigma\in{\left[n\right] \choose 2}|\ \sigma\leq_{\mathrm{lex}}\left(4,10\right)\right\}\) in \(O(n)\) time, for writing the at most $O(n)$ edges. Otherwise, it takes \(O(1)\) to calculate \(\Delta\left(T'\right)_{1}\), as $T'$ has bounded size.

    For step 5, we know that \(\left(3,n\right)\in\Delta\left(T\right)\) by $3$-hyperconnectivity, or~\ref{volume rigidity of torus kb and rp2}, and by \ref{vert split preserves crit regions} we have:
        \[\mathrm{Tail}_{\mathrm{lex}}\left(\Delta\left(T\right),\left( 5,6\right) \right)=\mathrm{Tail}_{\mathrm{lex}}\left(\Delta\left(T'\right),\left( 5,6\right) \right).\]
    This completely determines \(\Delta\left(T\right)_{1}\) (in \(O(n)\)); denote \[\ell:=10-\left|\mathrm{Tail}_{\mathrm{lex}}\left(\Delta\left(T\right),\left(5,6\right)\right)\right|.\]
    Then:
    \[\Delta\left(T\right)_{1}=\left\{ \sigma\in{\left[n\right] \choose 2}|\ \sigma\leq_{\mathrm{lex}}\left(4,\ell\right)\right\} \cup\mathrm{Tail}_{\mathrm{lex}}\left(\Delta\left(T\right),\left(5,6\right)\right).\]

    Step 6 is similar to step 5: let \(T_{p}\) be the triangulation reached after repeating steps 1-2 only for triangular critical disks, meaning \(T_{p}\) is the first prime triangulation reached from \(T\). Then by Theorem \ref{shifting of a large torus}, \(\Delta\left(T_{p}\right)_{2}\) is determined by the number of vertices in \(T_{p}\), and this determines \(\Delta\left(T\right)_{2}\) by \ref{contraction in triangle preserves 2nd dimension} (in $O(n)$ time):
    \[\Delta\left(T\right)_{2}\backslash\left(2,3,4\right)=\begin{cases}
\left\{ \sigma\in{\left[n\right] \choose 3}|\ \sigma\leq_{\mathrm{lex}}\left(1,4,8\right)\right\}  & \left|\left(T_{p}\right)_{0}\right|\geq8\\
\left\{ \sigma\in{\left[n\right] \choose 3}|\ \sigma\leq_{\mathrm{lex}}\left(1,4,7\right)\right\} \cup\left\{ \left(1,5,6\right)\right\}  & else
\end{cases}.\]

    This shows that the proposed algorithm determines \(\Delta\left(T\right)\) as claimed. All that is left to show is that steps 1-3 can be done in polynomial time.

    First, we present an algorithm for finding a reducible critical disk (step 1):\newline

    \underline{Step 1.A.}  For \(b\in \{3,4,5,6\}\), starting from $b=3$, do the following:
    
    For each simple loop \(L\) of length \(b\) in \(T'\) do steps i-iv, stopping when a reducible critical disk is found:

    \underline{Step i}. Find a connected component of \(\left|T'\right|\backslash\left|L\right|\), and calculate while doing so the Euler characteristic and the number of vertices of the component's closure, namely of the induced subcomplex of $T'$ on $L$ union the interior vertices on said component.
    
    \underline{Step ii}. Determine if \(L\) separates \(T'\), namely if the closure of the connected component from step i is not all of \(T'\), and if it does separate let \(D\) be the component whose closure's (non-reduced) Euler characteristic is \(1\); making \(D\) a disk.\footnote{The sum of Euler characteristics of connected components' closures equals the Euler characteristic of \(T\), which is 0. Therefore determining the Euler characteristic of one component's closure determines the other as well.}

    \underline{Step iii}. Determine if \(D\) is a combinatorial critical disk by checking if it has any diagonals, and checking if it has an internal vertex neighboring at least five boundary vertices, or has two neighboring internal vertices each neighboring four boundary vertices (see Proposition \ref{characterization of crit disks}).

    \underline{Step iv}. If \(D\) is a combinatorial critical disk with more than \(b-3\) internal vertices, we stop, as \(D\) is a reducible critical disk. Else, we repeat 1.A. for the next loop $L$.\newline

    Steps i-iv take \(O(n)\) time, and there are \(O(n^{b})\) simple loops of length \(b\) in \(T'\), meaning step 1.A. takes \(O(n^{b+1})\). Therefore, step \(1\) will take \(O(n^{7})\) altogether.

    In step 2, it takes \(O(n)\) time to find a contractible edge, as there are \(O(n)\) internal edges in \(D\) and checking if an edge is admissibly contractible according to  \ref{characterization of crit disks}(iii)  takes \(O(b)=O(1)\) time. Now, there are \(O(n)\) contractions to be done, meaning that step 2 takes \(O(n^{2})\). Therefore, steps 1-2 together take \(O(n^{7})\) and are repeated \(O(n)\) times in step 3, and therefore steps 1-3  take \(O(n^{8})\) altogether.

    We conclude that the algorithm finishes in \(O(n^{8})\) time, and we are done.
\end{proof}

\section{Concluding remarks}
Here are some directions for further research related to the results presented in  this paper.

\textbf{Higher genus}. An obvious direction is to extend to surfaces of higher genus our main results Theorems~\ref{alg exists} and~\ref{possible shiftings}, on polytime algorithms to compute, and on characterization of, the exterior shifting of surface triangulations. Showing volume rigidity here, namely \cite[Conj.5.1]{bulavka2023volume}, would be an important step. Another step is given in our Conjecture~\ref{conjecture on edges in shifting}. The algorithmic problem when the genus is unbounded might turn as hard as the general symbolic determinant identity testing problem (SDIT).
Note that it would be quite difficult to take a similar approach to what we did here even for the next highest genus, as one would have to check thousands of cases for the irreducible triangulations alone (\cite{sulanke2006generating}), let alone calculating sequences of vertex splits.

\textbf{Positive characteristic}. A second direction is to find analogues of our results for exterior shifting over nonzero characteristic.  Let us mention that for a fixed finite field, the SDIT problem is coNP-complete~\cite{BFS99}. 

\textbf{Symmetric shifting.}
A third direction is to find analogues of our results  for the symmetric shifting. 
A positive answer to Kalai's~\cite[Problem 3]{KalaiShifting}, together with the volume rigidity results~\cite[Thm.1.2, Cor.1.3]{bulavka2023volume}, would imply that the triangle $\{1,3,n\}$ belongs to the symmetric shifting of every $n$-vertex triangulation of either the torus, projective plane, or the Klein bottle.

\bibliography{refs}
\bibliographystyle{plain}

\appendix
\section{The Klein Bottle and \(\mathbb{RP}^{2}\)}
\label{appendix section}

\begin{proof}[Proof of Proposition \ref{characterization of crit disks} (ii), reducible case] We want to show that we can find an edge to contract that will not create a diagonal, or a vertex with five boundary neighbors\footnote{This is only possible for a pinched disk, as the Möbius strip has only four boundary vertices.}. In both cases we assume there is no internal vertex with four boundary neighbors, otherwise the problem reduces to contracting edges in a smaller critical disk, as in the proof of \ref{characterization of crit disks} (i) (see Figure \ref{fig:vertex 4 neighbors}).

    \underline{Möbius strip:} Let \(M\) be a triangulation of a Möbius strip obtained by gluing a hexagonal reducible critical disk to itself along an edge. This means that \(M\) has four boundary vertices, and at least four internal vertices (as here $b-3=3$), and a diagonal \(ab\) (for the glued edges on the haxagonal boundary of the disc). By \cite{MobiusLawrencenko}, this means that \(M\) must have a contractible internal edge \(\epsilon\), as all irreducible triangulations of the Möbius strip that have a diagonal have at least four boundary vertices and at most two internal vertices.

    If we can contract \(\epsilon\) without creating a diagonal, then we are done by \ref{vert split preserves crit regions} and induction on the number of internal vertices. If not, this means that there is an internal vertex \(u\) and edges \(cu,du\), where \(c,d\) are the other two boundary vertices (along with \(a,b\)), and \(\epsilon \in \{cu,du\}\); indeed, the diagonal created must connect $c$ with $d$. By internal connectedness, there is an internal vertex \(v\) and an edge \(uv\). If \(auv\) and \(buv\) are not missing triangles in \(M\), or if one of them is a missing triangle enclosing a disk, we are done, as either \(uv\) is contractible, or it encloses a triangular disk in which we can find a contractible edge, and these will not create a diagonal.

    Otherwise, without loss of generality \(auv\) is a missing triangle that does not enclose a disk, and therefore \(u\) does not neighbor \(b\) (otherwise it would neighbor all four boundary vertices, contradicting our assumption).
Now, either \(abcu\) or \(abdu\) is a loop in \(M\) that encloses a disk \(D\), and in \(D\) there must be some internal vertex \(w\) that neighbors \(b\) (since \(bu\),\(ac\) and \(ad\) are not edges in the disk \(D\)). Thus, either \(bw\) is contractible in \(M\), or it is part of a missing triangle that encloses a subdisk in \(D\), inside which we can find an edge to contract. See Figure \ref{fig:mobius alg proof} which illustrates this.

    \begin{figure}
        \centering
        \includegraphics{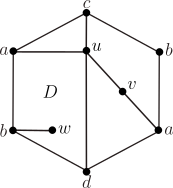}
        \caption{If \(uv\) cannot be contracted, we can find a contractible edge inside the disk \(D\)}
        \label{fig:mobius alg proof}
    \end{figure}

    \underline{Pinched disk:} Let \(P\) be a triangulation of a pinched disk obtained from a hexagonal reducible critical disk \(\widetilde{P}\) by identifying two opposite vertices \(a_{1},a_{2}\in \widetilde{P}\), and denote by \(a\in \partial P\) the corresponding vertex of \(P\). Then \(\partial \widetilde{P}\) is a loop \(a_{1}bca_{2}de\). As \(\widetilde{P}\) is reducible, both \(\widetilde{P}\) and \(P\) have at least four internal vertices.

    By the proof of \ref{characterization of crit disks} (i), there is some internal edge \(\epsilon\in \widetilde{P}\) that can be contracted in \(\widetilde{P}\) to obtain a critical disk with one less internal vertex. If this edge can be contracted in \(P\), we are done. Otherwise, \(\epsilon\) is part of a path \(a_{1}uva_{2}\) in \(\widetilde{P}\), with \(u,v\) internal. By internal connectedness, there is an edge \(uw\) (without loss of generality) with \(w\neq v\) internal. 
    \textbf{Case 1: \(uw\) is not contained in a missing triangle in \(P\)}. 
    If \(u,w\) do not collectively neighbor five boundary vertices, then we can contract \(uw\). If they do neighbor five boundary vertices, then we are in the leftmost case of the upper part of Figure \ref{fig:hexalg}, and we proceed as specified there.
    \textbf{Case 2: \(uw\) is part of a missing triangle in \(P\)}. If it is part of a missing triangle in \(\widetilde{P}\), then this missing triangle encloses a disk, and we can find a contractible edge inside this disk as per the proof of \ref{contractible edge exists in disk}. Otherwise, the missing triangle is \(auw\), and the corresponding path in \(\widetilde{P}\) is \(a_{1}uwa_{2}\).

    In this last case, we turn our attention to the disk \(D\) in \(P\) bounded by the loop \(awuv\). Since \(ua_{2}\) is not an edge in \(\widetilde{P}\), there is some edge \(\varepsilon = wx\) in \(D\) but not in \(\partial D\) (see Figure \ref{fig:pinched disk}). 

    \begin{figure}
        \centering
        \includegraphics{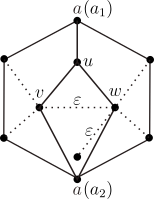}
        \caption{If \(\varepsilon = vw\) and \(v,w\) collectively neighbor 5 boundary vertices, we are in one of the cases of Figure \ref{fig:hexalg}.}
        \label{fig:pinched disk}
    \end{figure}


    If $x\neq v$ and \(\varepsilon\) is in a missing triangle in \(P\), this missing triangle encloses a disk inside $D$, hence we can find an edge to contract in that disc. If, on the other hand \(\varepsilon\) is not part of a missing triangle, then if \(w,x\) do not collectively neighbor five boundary vertices, then we can contract \(\varepsilon\). If they do, then \(x=v\) and we are in the bottom case of the upper part of Figure \ref{fig:hexalg} (both $w$ and $v$ neighbor the boundary vertex $a_2$ in \(\widetilde{P}\)), and we proceed as specified there, and we are done.
\end{proof}

    We can treat the projective plane and the Klein bottle in a fashion similar to the torus.

    For the projective plane, its reduced homology is trivial (over the reals), and therefore, by \ref{Shifting preserves Betti numbers}, the shifting will be a pure 2-dimensional complex (which means that the 2-faces of the shifting completely determine the edges). 
    The following characterization will be used to prove both~\ref{possible shiftings}(ii) and~\ref{alg exists} for the projective plane.
  
    \begin{theorem}
    \label{characterization of shifting of rp2}
        If \(T\) is a prime triangulation of \(\mathbb{RP}^{2}\) with \(n\) vertices, then \(\Delta\left(T\right)_{2}\) is a lexicographic prefix, and in particular:
        \[\mathrm{max}_{\mathrm{lex}}\Delta\left(T\right)_{2}=\begin{cases}
\left\{ 1,4,7\right\}  & n\geq7\\
\left\{ 1,5,6\right\}  & n<7
\end{cases},\]
Where \(\mathrm{max}_{\mathrm{lex}} S\) is the maximal element of \(S\) w.r.t. the lexicographic ordering, for \(\emptyset \neq S\subseteq{\left[n\right] \choose d}\).
    \end{theorem}
    \begin{proof}
        There are only two irreducible triangulations of \(\mathbb{RP}^{2}\), one with six vertices and one with seven (\cite{BarnetteRP2}). In particular there is only one triangulation of \(\mathbb{RP}^{2}\) with less than seven vertices. Calculation then shows that the theorem holds for all cases with \(n=6,7\). We then get the theorem in all other cases using the exact same argument as in the proof of \ref{shifting of a large torus}(\ref{eq: two-faces of shifted torus}).   
    \end{proof}

    For the Klein bottle, the situation is slightly more complicated. The following characterization will be used to prove both~\ref{possible shiftings}(iii) and~\ref{alg exists} for the Klein bottle.
    \begin{theorem}
    \label{characterization of shifting of klein bottle}
        Let \(K\) be a triangulation of a Klein bottle with \(n\) vertices.
        
        (i) If \(K\) can be written as \((P \cup P')\backslash \sigma\), where \(P,P'\) are triangulations of \(\mathbb{RP}^{2}\) such that \(P \cap P'\) is a 2-face \(\sigma\) (and all its subsets), then: \[\left\{ 1,5,6\right\} \in\Delta\left(K\right)\iff\left\{ 1,5,6\right\} \in\Delta\left(P\right)\cap\Delta\left(P'\right).\]

        (ii) If \(K\) cannot be written as in (i), and if \(K\) has no reducible (combinatorial) critical disk with 3 or 4 boundary vertices, then: \[\left\{ 1,5,6\right\} \in\Delta\left(K\right)\iff n=8.\]

        (iii) \(\left\{ 1,5,7\right\} \notin\Delta\left(K\right)\) (for every triangulation \(K\)).

        (iv) If \(n\geq 13\) and \(K\) has no reducible combinatorial critical regions, then \(\left\{ 5,6\right\} \notin\Delta\left(K\right)\).
    \end{theorem}
    \begin{proof}
        Again, we calculate for small cases and then infer to the larger cases (using the list of irreducible triangulations in \cite{LawrencenkoNegamiKlein} and \cite{SulankeNoteKlein}):

        \underline{(i):} If \(\left\{ 1,5,6\right\} \in\Delta\left(P\right)\cap\Delta\left(P'\right)\), then by \ref{Shifting a union over a simplex}: \[\left\{ 1,5,7\right\} \in\Delta\left(P\cup P'\right).\] Therefore removing one face from \(P\cup P'\) will not remove \(\left\{ 1,5,6\right\}\) from the shifting, as it is not maximal in \(\Delta\left(P\cup P'\right)\) (w.r.t. \(\leq_{p}\)).

        Now assume that \(\left\{ 1,5,6\right\} \notin\Delta\left(P\right)\). Then by \ref{prime triangulation survival} and \ref{characterization of shifting of rp2}, we can contract edges in \(K\) until \(P\) becomes a prime triangulation with seven vertices, and \(P'\) becomes prime with six or seven vertices (we will never contract an edge of the common simplex \(\sigma\), and thus each edge contraction can be seen as occurring in \(P\) or in \(P'\), but not in both). Thus by \ref{vert split preserves tails} the problem reduces to calculating for all the cases where \(P\) is prime with seven vertices and \(P'\) is prime with six vertices (since if \(P'\) also has seven vertices and is prime, then \(\left\{ 1,5,6\right\}\) cannot be in \(\Delta\left(K\right)\) by \ref{Shifting a union over a simplex}, as it is not in \(\Delta\left(P\right)\) nor in \(\Delta\left(P'\right)\) by \ref{characterization of shifting of rp2}).
        
        \underline{(ii):} By calculation, this holds for all irreducible triangulations for which the condition of (i) does not hold; note that all irreducible triangulations of the Klein bottle have at least eight vertices. Moreover, this holds for all critically irreducible triangulations with nine vertices, and the only critical regions in triangulations with eight vertices are triangular and quadrilateral critical disks. Therefore, assertion (ii) follows via the same argument as in the proof of \ref{shifting of a large torus}.
        
        \underline{(iii):} This holds for all irreducible triangulations (by calculation), and thus for any triangulation by \ref{vert split preserves tails}.

        \underline{(iv):} The proof follows the argument used to prove~\ref{shifting of a large torus}(\ref{eq:edges of shifted torus}), applied to the Klein bottle: for every sequence of non-critical splits \(K_{0}\rightarrow K_{1}\rightarrow\cdots\rightarrow K_{r}\) with \(K_{0}\) irreducible and \((K_{r})_{0} = 13\), then \(\left\{ 5,6\right\} \notin\Delta\left(K_{r}\right)\) by calculation, and furthermore if \(0\leq i \leq r\) is minimal such that \(\left\{ 5,6\right\} \notin\Delta\left(K_{i}\right)\), then for all \(0 \leq j < i\), all critical regions in \(K_{j}\) are combinatorial. The argument is then the same as in the proof of~\ref{shifting of a large torus}(\ref{eq:edges of shifted torus}).
        
    \end{proof}

\begin{proof}[Proof of Theorem~\ref{possible shiftings} (ii) and (iii), Sketch.]
We proceed similar to the proof of \ref{possible shiftings}(i), using \ref{characterization of shifting of rp2}, resp. \ref{characterization of shifting of klein bottle}, instead of~\ref{shifting of a large torus}.
\end{proof}

\begin{proof}[Proof of Theorem~\ref{alg exists} for the projective plain and the Klein bottle, Sketch.]
As before, we proceed similar to the proof of \ref{alg exists} for the torus, using \ref{characterization of shifting of rp2}, resp. \ref{characterization of shifting of klein bottle}, instead of~\ref{shifting of a large torus}, for getting a polytime algorithm to compute the exterior shifting.

For \(\mathbb{RP}^{2}\), we only need to contract in triangular critical disks in order to get a prime triangulation, so in fact the algorithm we get is much simpler than the one for the torus, and it will finish in \(O(n^{5})\).



    For the Klein bottle the algorithm is similar but slightly more complicated.
    This time, when searching for critical regions, a loop may separate the triangulation into two Möbius strips instead of enclosing a disk, and we also need to search for pinched disks whose boundary has five vertices. Therefore, when determining which connected component has the topology we are interested in, we need to check orientability in addition to the Euler characteristic (which can also be done in \(O(n)\) time). If at some point we find a loop of length 3 which separates two Möbius strips (so each of them is a projective plane minus a disc), then the problem reduces to calculating the shifting of two projective planes (by \ref{characterization of shifting of klein bottle} (i) and \ref{Shifting a union over a simplex}). Also, if we find a Möbius strip with four boundary vertices, then we need to check that it has exactly one diagonal for it to be a critical region (as opposed to the disk and pinched disk where we want no diagonals). All of these will not contribute anything significant to the complexity, and the algorithm will finish in \(O(n^{8})\) time, like the torus.
\end{proof}   
\end{document}